\documentclass[a4paper, 12pt]{article}
\usepackage[utf8]{inputenc}
\usepackage[top=2.5cm,bottom=3.0cm,left=2.5cm,right=2.5cm]{geometry}
\usepackage{amssymb,amsmath,amsfonts,amsthm}
\usepackage{graphics}
\usepackage{graphicx}  

\usepackage{color}
\usepackage{soul}

\usepackage{lipsum}

\usepackage{cancel}

\usepackage[pagewise]{lineno} 

\providecommand{\keywords}[1]
{
	\small	
	\textbf{\textit{Keywords---}} #1
}

\newtheorem{theorem}{Theorem}[section]
\newtheorem{lemma}[theorem]{Lemma}

\newtheorem{corollary}[theorem]{Corollary}
\theoremstyle{definition}

\newtheorem{remark}[theorem]{Remark}
 
\newtheorem*{theorem*}{Theorem} 
\newtheorem*{definition*}{Definition}
\newtheorem*{prop*}{Proposition}
\newtheorem*{corollary*}{Corollary}

\newcommand{\C}{\mathbb{C}}

\newcommand{\N}{\mathbb{N}}

\newcommand{\Z}{\mathbb{Z}}

\newcommand{\SP}{\mathbb{S}}

\newcommand{\bm}{\mathbf}

\newcommand{\sub}{\subseteq}

\newcommand{\lav}{\left|}
\newcommand{\rav}{\right|}

\newcommand{\ldav}{\left| \left|}
\newcommand{\rdav}{\right| \right|}

\newcommand{\oo}{\infty}
\newcommand{\lc}{\left(} 
\newcommand{\rc}{\right)}
\newcommand{\lb}{\left[}
\newcommand{\rb}{\right]}
\newcommand{\lfp}{\left\{}
\newcommand{\rfp}{\right\}}
\newcommand{\mc}{\mathcal}

\usepackage{hyperref}
\hypersetup{
	colorlinks,
	citecolor=black,
	filecolor=black,
	linkcolor=black,
	urlcolor=black
}

\title{Dynamical Properties of Perturbed Weighted Shifts} 
\author{A. Messaoudi, J. L. Tofanin Neto, I. Tsokanos}
\date{ }

\begin{document}

	\maketitle

\begin{abstract}

In this paper, sufficient conditions are established for diagonal perturbations of weighted shifts to be supercyclic, hypercyclic, and frequently hypercyclic, extending several classical results for weighted shifts to this more general setting. In addition, a complete analytic characterization of the shadowing property for these operators is obtained. Consequently, examples of frequently hypercyclic perturbed weighted shifts with the shadowing property are constructed.

\end{abstract}


\smallskip 
\noindent


\section{Introduction}

A \textit{linear dynamical system} is a pair $(\mathcal{B},T)$, where $\mathcal{B}$ is a Banach space, or more generally, a Fréchet space, and $T:\mathcal{B}\to\mathcal{B}$ is a bounded linear operator. A central objective is to understand the asymptotic behavior of orbits. More precisely, for $\bm{v}\in\mathcal{B}$, its orbit is the set
$O_{T}\lc\bm{v}\rc=\lfp T^{n}\lc\bm{v}\rc:n\in\mathbb{Z}_{+}\rfp,$
where $\mathbb{Z}_{+}$ denotes the nonnegative integers and $T^{n}$ the $n$-th iterate of $T$. Over the past three decades, the theory of linear dynamics has developed substantially, with deep connections to ergodic theory, number theory, and the geometry of Banach spaces; see, for instance, the books \cite{Bayart-Matheron-Dynamics_Of_Linear_Operators, Brezis_Functional_Analysis} and the papers \cite{Bayart-Ruzsa_difference_sets, Bourhim-Spectrum_of_Bilateral_Shifts, Cirilo_Gollobit_Pujals-Dynamics_of_Generalized_Hyperbolic_Linear_Operators, Kitai_invariant_closed_sets}.

\paragraph{}

In this work, the theory of chaotic and shadowing properties for unilateral weighted backward shifts is extended to perturbed weighted shifts. In what follows, given 
two bounded complex sequences $\bar d=(d_k)_{k\ge1}$ and $\bar w=(w_k)_{k\ge1}$ with $\inf_{k\ge1}\lav w_k\rav\neq0$, the associated \textit{perturbed weighted shift} is the operator $T_{\bar d,\bar w}:\ell^{p}(\mathbb N)\longrightarrow\ell^{p}(\mathbb N)$ given by 
\begin{equation}\label{Eq_Perturbed Weighted Shift Definition} 
T_{\bar d,\bar w}=I_{\bar d}+B_{\bar w},
\quad \text{where } B_{\bar w}\lc(x_k)_{k=1}^{+\oo}\rc
=\bigl(w_kx_{k+1}\bigr)_{k=1}^{+\oo} \quad \text{and} \quad 
I_{\bar d}\lc(x_k)_{k=1}^{+\oo}\rc
=(d_kx_k)_{k=1}^{+\oo}.
\end{equation}
Here, $B_{\bar w}$ is the \textit{unilateral weighted backward shift}, $I_{\bar d}$ is the diagonal perturbation and $\ell^{p}(\mathbb N)$ ($1\le p<+\infty$) denotes the Banach space of complex $p$-summable sequences equipped with its usual $p$-norm $\ldav\cdot\rdav_{p}$. Although the perturbed weighted shift is a natural and simple generalization of weighted shifts, studying its dynamic properties is more complex than in the case of weighted shifts

\paragraph{}
A bounded linear operator $T$ on a separable Banach space $\mc{B}$ is said to be \emph{hypercyclic} if it admits a dense orbit in $\mc{B}$. Equivalently, $T$ is \emph{topologically transitive}: for every pair of non-empty open subsets $U,V\subseteq\mc{B}$, there exists $n\geq0$ such that $T^{n}(U)\cap V\neq\varnothing$. The operator $T$ is \emph{frequently hypercyclic} if there exists $\bm{v}\in\mathcal{B}$ such that, for every non-empty open subset $U\subseteq\mathcal{B}$, the set
$N\lc \bm{v},U\rc=\lfp n\in\mathbb{Z}_{+}:T^{n}\lc\bm{v}\rc\in U\rfp$
has positive lower density, that is,
$\underset{n\to+\oo}{\liminf} \,n^{-1}\cdot|N\lc\bm{v},U\rc\cap\{0,\ldots,n-1\}|>0$.
Finally, $T$ is \emph{supercyclic} if there exists $\bm{v}\in\mathcal{B}$ whose projective orbit
$\lfp\lambda T^{n}\lc\bm{v}\rc:n\in\mathbb{Z}_{+},\ \lambda\in\mathbb{C}\rfp$
is dense in $\mathcal{B}$. These three properties are important topics in linear dynamics; see, for instance,
\cite{Bayart_Grivaux_2006, Bayart-Matheron-Dynamics_Of_Linear_Operators,
Bayart-Ruzsa_difference_sets}.

Hilden and Wallen \cite{Hilden_Wallen-Some_Cyclic_and_Non-Cyclic_Vectors_of_Certain_Operators} showed that every unilateral weighted shift $B_{\bar{w}}$ with nonzero weights is supercyclic, whereas necessary and sufficient conditions for hypercyclicity and frequent hypercyclicity were established by Salas \cite{Salas-Hypercyclic_Weighted_Shifts} and Grosse-Erdmann \cite{GE}, and by Bayart and Ruzsa \cite{Bayart-Ruzsa_difference_sets}, respectively. 
More precisely, a unilateral weighted shift $B_{\bar w}$ on $\ell^{p}\lc \N \rc$ with nonzero weights is hypercyclic (respectively, frequently hypercyclic) if and only if
\begin{equation}\label{Eq_Hypercyclicity Weighted Shift}
\limsup_{k \to\infty}|w_{1}\cdots w_{k}|=\infty \qquad \lc \text{respectively}, \quad \sum_{k=1}^{\infty}\frac{1}{|w_{1}\cdots w_{k}|^{p}}<\infty \rc. 
\end{equation} 

Regarding the chaotic behavior of perturbed weighted shifts, Salas \cite{Salas-Hypercyclic_Weighted_Shifts} proved that the operator $I+B_{\bar w}$ on $\ell^{2}(\mathbb N)$ is hypercyclic whenever $\bar w$ is a bounded sequence of positive weights. Grivaux \cite{Gr2} subsequently established that these operators are mixing, while Saavedra and Rodríguez \cite{SR} showed that the operator $I+B_{\bar w}$ satisfies the Hypercyclicity Criterion. Complementary to Salas' result, Grivaux \cite{Grivaux-Sums_of_hypercyclic_operators} studied operators of the form $T_{\bar d,w}=I_{\bar d}+B_w$ on $\ell^{2}(\mathbb N)$, where the diagonal part is allowed to vary while the shift has constant weight $w\in\mathbb C$ (see also \cite{Grivaux_Matheron_Menet-Linear_Dynamical_Systems_on_Hilbert_Spaces}). In particular, $T_{\bar d,w}$ is hypercyclic whenever 
\begin{equation}\label{Eq_Grivaux Sufficient Condition}
|w|>\max\lfp  2||I_{\bar d}||, \; 1+||I_{\bar d}|| \rfp, \quad \text{where } ||I_{\bar d}||=\sup\lfp \lav d_{k} \rav : \; k \ge 1 \rfp . 
\end{equation}

\paragraph{} 
The shadowing property is another fundamental notion in dynamical systems. In the context of linear dynamics, a bounded linear operator $T$ is said to have the \textit{shadowing property} if, for every $\varepsilon>0$, there exists $\delta>0$ such that every $\delta$-pseudo-orbit $\left(\bm{x}_m\right)_{m\in\mathbb{N}}$, that is, a sequence in $\mc{B}$ satisfying  $\| \bm{x}_{m+1}-T \lc \bm{x}_m \rc\|\leq\delta$, $ m\in\mathbb{N}$,  
is $\varepsilon$-shadowed by an orbit; there exists a vector $\bm{v} \in\mathcal{B}$ such that $ \|\bm{x}_m - T^m\lc \bm{v} \rc \|\leq\varepsilon$, for every $m\in\mathbb{N}$.
If $T$ is invertible, then $\mathbb{N}$ may equivalently be replaced by $\mathbb{Z}$; see \cite[Theorem~1]{Bernardes_Peris-On_Shadowing_and_Chain_Recurrence_in_Linear_Dynamics}.

Bernardes and Messaoudi \cite{Bernardes_Messaoudi-Shadowing_and_Structural_Stability} proved that $B_{\bar w}$ has the shadowing property if and only if one of the following conditions holds:
\begin{equation}\label{Eq_Weighted Shifts Shadowing Characterization}
\lim_{l\to\infty}\sup_{k\ge0}|w_{k+1}\cdots w_{k+l}|^{1/l}<1
\qquad\text{or}\qquad
\lim_{l\to\infty}\inf_{k\ge0}|w_{k+1}\cdots w_{k+l}|^{1/l}>1.
\end{equation}

\paragraph{}

The relation between chaotic and shadowing behavior of linear operators has been studied only in recent years. For instance, Antunes et al. \cite{Antunes} proved that an operator on a Banach space possessing both the shadowing property and chain recurrence is frequently hypercyclic and mixing.  
A broad class of frequently hypercyclic operators possessing the shadowing property can be found in \cite{BM,Cirilo_Gollobit_Pujals-Dynamics_of_Generalized_Hyperbolic_Linear_Operators}.

\paragraph{}

This paper establishes sufficient conditions for diagonal perturbations of weighted shifts to be supercyclic, hypercyclic, and frequently hypercyclic, thereby extending several classical results for weighted shifts to this class of operators. In particular, the criteria recover and generalize previous results of Salas, Bayart--Ruzsa, and Grivaux, while also yielding new families of frequently hypercyclic operators on $\ell^{p}(\mathbb N)$. Moreover, a characterization of the shadowing property for perturbed weighted shifts $T_{\bar d,\bar w}$ on $\ell^p(\mathbb N)$ is obtained, extending the existing characterization \eqref{Eq_Weighted Shifts Shadowing Characterization} for unilateral weighted shifts to diagonal perturbations of weighted shifts. As a consequence, examples of perturbed weighted shifts that are both frequently hypercyclic and have the shadowing property are obtained.

The paper is organized as follows. Section \ref{Preliminaries} contains the necessary preliminaries and introduces the notation used throughout the paper. The main results are stated in Section \ref{Sec_Main Results}, and their proofs are given in Sections \ref{Sec_Proof of Hypercyclicity} and \ref{section 5}.

\section{Preliminaries}\label{Preliminaries}

Given a Banach space $\mathcal{B}$ and a bounded linear operator $T:\mathcal{B}\to\mathcal{B}$, the operator norm of $T$ is denoted by $ \|T\|=\sup\lfp\|Tx\|:x\in\mc{B},\ \|x\|=1\rfp$. 
The \emph{spectrum}, \emph{point spectrum}, \emph{approximate point spectrum}, and \emph{surjective spectrum} of $T$ are defined by 
\begin{align*} 
\sigma(T) &= \{\lambda\in\mathbb C:T-\lambda I\text{ is not invertible}\},\\
\sigma_{pt}(T) &= \{\lambda\in\mathbb C:T-\lambda I\text{ is not injective}\},\\ 
\sigma_{ap}(T) &=  \lfp \lambda\in\mathbb C:T-\lambda I\text{ is not bounded below} \rfp,\\ 
\sigma_{sur}(T)&=\lfp \lambda\in\C:\; T-\lambda I \text{ is not surjective} \rfp, 
\end{align*}

\paragraph{}

In Section~\ref{Sec_Proof of Hypercyclicity}, the Supercyclicity Criterion, the Hypercyclicity Criterion, and the Frequent Hypercyclicity Criterion (see, respectively, \cite{Salas-Supercyclicity_and_Weighted_Shifts}, \cite{Kitai_invariant_closed_sets}, and \cite{Bayart_Grivaux_2006}) will be used to establish sufficient conditions for a perturbed weighted shift $T_{\bar d,\bar w}$ to be supercyclic, hypercyclic, and frequently hypercyclic. These criteria are fundamental tools in linear dynamics. The formulation below is tailored to the present setting and suffices for the purposes at hand.

\begin{theorem}[Supercyclicity, Hypercyclicity, and Frequent Hypercyclicity Criteria {\cite{Bayart_Grivaux_2006, Bayart-Matheron-Dynamics_Of_Linear_Operators, Bayart-Ruzsa_difference_sets}}]
\label{Theor_SHFH Criterion}

Let $\mathcal B$ be a separable Banach space and $T:\mathcal B\to\mathcal B$ be a continuous linear operator. Assume that there exist a dense subset $D_0\subseteq\mathcal B$ and a map $S:D_0\to D_0$ such that
\begin{equation}\label{Eq_SHFH Criterion Inverse}
T\circ S( \bm{x} ) = \bm{x} ,
\qquad
\text{for every } \bm{x} \in D_0.
\end{equation}
Then the following assertions hold.

\begin{enumerate}

\item
If there exists a subsequence $(m_n)_{n\ge1}$ such that, for every $\bm{x} , \bm{y} \in D_0$,
\begin{equation}\label{Eq_S Criterion Limit Zero}
\lim_{n\to\infty}
\|T^{m_n} \bm{x} \|\,
\|S^{m_n} \bm{y} \|
=
0,
\end{equation}
then $T$ is supercyclic.

\item
If there exists a subsequence $(m_n)_{n\ge1}$ such that, for every $x\in D_0$,
\begin{equation}\label{Eq_H Criterion LimInf Condition}
\lim_{n\to\infty}\|T^{m_n} \bm{x} \|=0
\quad
\text{and}
\quad
\lim_{n\to\infty}\|S^{m_n} \bm{x} \|=0,
\end{equation}
then $T$ is hypercyclic. 

\item
If, for every $x\in D_0$, the series
\begin{equation}\label{Eq_FH Criterion Convergent Series Condition}
\sum_{m=0}^{\infty}T^m \bm{x} 
\qquad\text{and}\qquad
\sum_{m=0}^{\infty}S^m \bm{x} 
\end{equation}
converge unconditionally in $\mathcal B$, then $T$ is frequently hypercyclic.

\end{enumerate}

\end{theorem}

To formulate the sufficient conditions for supercyclicity, hypercyclicity, and frequent hypercyclicity of perturbed weighted shifts, the following notation will be used. Given the sequence $\bar d=\lc d_k\rc_{k\ge1}$, for $r\in \mathbb{Z}_{+}$ and $s\in\Z$, define 
\begin{equation}\label{Eq_D Quantities Definition} 
		D_{r} \lc s \rc =   \begin{cases}
			0, & \text{if } s < 0 \text{ or } r<s, \\
			1, & \text{if } s=0, \\
			\displaystyle     \lc -1 \rc^{s}   \sum_{1 \le i_{1} < \dots < i_{s} \le r} d_{i_{1}} \cdots d_{i_{s}}  , & \text{if } 1 \le s \le r  \\
		\end{cases}
\end{equation} 
and 
\begin{equation}\label{Eq_Q Quantities Definition}
    Q_{r}\lc s \rc =  \begin{cases}
        0, & \text{if } s < 0, \\ 
        1 ,& \text{if } s = 0, \\ 
        \underset{\substack{s_{1} + \dots + s_{r} = s , \\ s_{1}, \dots, s_{r} \ge 0 }}{\sum} d_{1}^{s_{1} } \cdots d_{r}^{s_{r}}, & \text{if } s \ge 1 \\ 
    \end{cases}   
\end{equation} 
The quantities $D_r\lc s\rc$ and $Q_r\lc s\rc$ are the \textit{signed elementary symmetric polynomials} and \textit{complete symmetric polynomials}, respectively, in the variables $d_1,\dots,d_r$. Their generating functions, understood as formal power series, are
\begin{equation}\label{Eq_Generating Functions of D and Q}
\prod_{k=1}^{r} \lc 1-d_{k} t \rc
=\sum_{s=0}^{r}D_r\lc s\rc t^s
\quad\text{and}\quad
\prod_{k=1}^{r}\lc 1-d_{k} t \rc^{-1}
=\sum_{s=0}^{+\oo}Q_r\lc s\rc t^s,
\end{equation}
see for instance \cite{Macdonald-Symmetric_Functions_and_Hall_Polynomials}.


\section{Main Results}\label{Sec_Main Results}

The first main result establishes sufficient conditions under which a non-injective perturbed weighted shift $T_{\bar{d}, \bar{w}}$, determined by sequences of complex numbers $\bar d=(d_k)_{k\ge1}$ and $\bar w=(w_k)_{k\ge1}$, is supercyclic, hypercyclic, or frequently hypercyclic. To state the result, define
\begin{equation}\label{Eq_W Product} 
	   W_{k} = \begin{cases} 
		1 , & \text{if } k = 0, \\  
		w_{1} \cdots w_{k}, & \text{if } k  \ge 1 .\\ 
	\end{cases}  
\end{equation}

\begin{theorem}\label{Theor_H and FH for Perturbed Weighted Shifts}
Let $T_{\bar d,\bar w}:\ell^{p}(\mathbb N)\to\ell^{p}(\mathbb N)$ be a bounded perturbed weighted shift. Assume that
\begin{equation}\label{Eq_Bounded Inverse Vector}
\sum_{k=1}^{\infty}
\frac{\left|D_{k-1}(k-m)\right|^p}
     {\left|W_{k-1}\right|^p}
<\infty,
\qquad
\text{for every }m\ge1,
\end{equation}
and that, for every $n\ge1$, the series of vectors 
\begin{equation}\label{Eq_Dense Kernels Sufficient Condition}
\sum_{m=1}^{\infty} \lc Q_n(m-n)\,  \lc \dfrac{ D_{k-1}\lc k-m \rc } {W_{k-1}}  \rc_{k=1}^{+\oo} \rc
\quad\text{converges in }\ell^{p}(\mathbb N),
\end{equation}
where the quantities $D_{r}\lc s \rc$, $Q_{r}\lc s\rc$, and $W_{k}$ are defined in \eqref{Eq_D Quantities Definition}, \eqref{Eq_Q Quantities Definition}, and \eqref{Eq_W Product}, respectively. 
Then $T_{\bar d,\bar w}$ is supercyclic. 
Moreover, the following assertions hold:

\begin{enumerate}
\item Assume, furthermore, that 
\begin{equation}\label{Eq_Hypercyclic Sufficient Condition Liminf}
 \liminf_{m\to \infty} \sum_{k=1}^{\infty}
\frac{\left|D_{k-1}(k-m)\right|^p}
     {\left|W_{k-1}\right|^p}
=0.
\end{equation}
Then the operator $T_{\bar{d}, \bar{w}}$ is hypercyclic. 
Moreover, if $T_{\bar{d}, \bar{w}}$ is hypercyclic and the vectors $ \lc W_{m-1} \lc \dfrac{ D_{k-1}\lc k-m \rc } {W_{k-1}}  \rc_{k=1}^{+\oo} \rc_{m\ge1}$ form a Schauder basis of $\ell^{p}(\mathbb N)$, then condition~\eqref{Eq_Hypercyclic Sufficient Condition Liminf} holds.

\item Assume, furthermore, that the series of vectors 
\begin{equation}\label{Eq_Freq Hypercyclic Sufficient Condition Series Convergence}
\sum_{m=1}^{\infty} \lc \dfrac{ D_{k-1}\lc k-m \rc } {W_{k-1}}  \rc_{k=1}^{+\oo} 
\quad\text{converges unconditionally in }\ell^{p}(\mathbb N) . 
\end{equation}
Then the operator $T_{\bar{d}, \bar{w}}$ is frequently hypercyclic and $ \SP^{1} \sub \sigma_{pt}\lc T_{\bar{d}, \bar{w}} \rc$. 
In particular, if 
\begin{equation}\label{Eq_Strong Condition for FH}
\sum_{k=1}^{+\infty}
\left|
\frac{\prod_{j=1}^{k}(1+|d_j|)}
     {w_1\cdots w_k}
\right|^p
<+\infty,
\end{equation}
then $T_{\bar d,\bar w}$ is frequently hypercyclic.
\end{enumerate}
\end{theorem} 

To facilitate the interpretation of the hypotheses of Theorem~\ref{Theor_H and FH for Perturbed Weighted Shifts}, denote the vectors appearing in the statement of the theorem by
\begin{equation}\label{Eq_Kernel Basis Vectors} 
\bm{v}_{m}
:=
\lc
\dfrac{D_{k-1}\lc k-m\rc}{W_{k-1}}
\rc_{k=1}^{+\oo},
\qquad
m\ge1.
\end{equation}
For instance,
$$
\bm{v}_{1}
=
\lc
1,
-\dfrac{d_{1}}{w_{1}},
\dfrac{d_{1}d_{2}}{w_{1}w_{2}},
-\dfrac{d_{1}d_{2}d_{3}}{w_{1}w_{2}w_{3}},
\dots
\rc, \quad 
\bm{v}_{2}
=
\lc
0,
\dfrac{1}{w_{1}},
-\dfrac{d_{1}+d_{2}}{w_{1}w_{2}},
\dfrac{d_{1}d_{2}+d_{1}d_{3}+d_{2}d_{3}}{w_{1}w_{2}w_{3}},
\dots
\rc,
$$
and, in general,
$$
\bm{v}_m
=
\lc
\underbrace{0,\ldots,0}_{m-1},
\dfrac{1}{w_1\cdots w_{m-1}},
-\sum_{i=1}^{m}\dfrac{d_i}{w_1\cdots w_m},
\sum_{1\le i_1<i_2\le m+1}
\dfrac{d_{i_1}d_{i_2}}{w_1\cdots w_{m+1}},
\ldots
\rc.
$$
In terms of these vectors, condition~\eqref{Eq_Bounded Inverse Vector} ensures that $\bm{v}_m\in\ell^p(\mathbb N)$ for every $m\ge1$, whereas condition~\eqref{Eq_Dense Kernels Sufficient Condition} requires the convergence in $\ell^p(\mathbb N)$ of
$\sum_{m=1}^{\infty}Q_n(m-n)\bm{v}_m$ for every $n\ge1$.

\begin{remark}
    In the proof of the theorem, frequent hypercyclicity is established by applying the Frequent Hypercyclicity Criterion. Every operator satisfying this criterion is also Devaney chaotic; that is, it is hypercyclic and has a dense set of periodic points.
\end{remark}

\begin{remark}
The perturbed weighted shifts $T_{\bar d,\bar w}$ covered by Theorem~\ref{Theor_H and FH for Perturbed Weighted Shifts} are non-injective. Indeed, condition~\eqref{Eq_Bounded Inverse Vector} implies that $\bm{v}_{1}\in\ell^{p}\lc\N\rc$, and it is readily verified that $\bm{v}_{1}$ spans the kernel of $T_{\bar d,\bar w}$.
\end{remark}

\begin{remark}
When $\bar d$ is the zero sequence $\bar{\mathbf 0}=(0)_{n\ge1}$, then
$$\bm v_m=\lc
\underbrace{0,\ldots,0}_{m-1},
\dfrac{1}{w_1\cdots w_{m-1}},
0,0,\ldots
\rc$$
and Theorem~\ref{Theor_H and FH for Perturbed Weighted Shifts} reduces to the classical results for unilateral weighted shifts. In particular, it recovers the fact that every unilateral weighted shift is supercyclic \cite{Salas-Supercyclicity_and_Weighted_Shifts}, as well as the classical criteria of Salas, Bayart, and Ruzsa for hypercyclicity and frequent hypercyclicity \cite{Salas-Hypercyclic_Weighted_Shifts,Bayart-Ruzsa_difference_sets}. Indeed, in this case, assumptions~\eqref{Eq_Bounded Inverse Vector} and~\eqref{Eq_Dense Kernels Sufficient Condition} are automatically satisfied, while conditions~\eqref{Eq_Hypercyclic Sufficient Condition Liminf} and~\eqref{Eq_Strong Condition for FH} yield precisely the classical criteria of Salas and Bayart--Ruzsa stated in~\eqref{Eq_Hypercyclicity Weighted Shift}.
\end{remark}

\paragraph{}

The following corollary of Theorem~\ref{Theor_H and FH for Perturbed Weighted Shifts} provides a convenient sufficient condition for $T_{\bar d,\bar w}$ to be frequently hypercyclic. Let
\begin{equation}\label{Eq_Cor Inverse Operator Norm Bound}
c_{\bar d,\bar w}
:=
\sup_{k\ge1}\frac{1}{|w_k|}
+
\sum_{l=1}^{\infty}
\sup_{k\ge1}
\left|
\frac{d_{k+1}\cdots d_{k+l}}
{w_k w_{k+1}\cdots w_{k+l}}
\right|.
\end{equation}
Recall that $\ldav I_{\bar d} \rdav =\sup_{k\ge1}|d_k|$.

\begin{corollary}\label{Cor_Perturbed Weighted Shift Sufficient Condition}
Let $T_{\bar d,\bar w}:\ell^{p}(\mathbb N)\to\ell^{p}(\mathbb N)$ be a perturbed weighted shift. If
\begin{equation}\label{Eq_Cor Perturbed Weighted Shift Sufficient Condition}
c_{\bar d,\bar w}
<
\min\lfp \frac{1}{ \ldav I_{\bar d} \rdav } , 1 \rfp,
\end{equation}
then $T_{\bar d,\bar w}$ is frequently hypercyclic. In particular, if
\begin{equation}\label{Eq_Cor Perturbed Weighted Shift Strong Sufficient Condition}
\max\lfp
2\sup_{k\ge1}\lav d_k\rav,
1+\sup_{k\ge1}\lav d_k\rav
\rfp
<
\inf_{k\ge1}\lav w_k\rav,
\end{equation}
then $T_{\bar d, \bar w}$ is frequently hypercyclic. 
\end{corollary}

Corollary~\ref{Cor_Perturbed Weighted Shift Sufficient Condition} extends Grivaux's result \cite{Grivaux-Sums_of_hypercyclic_operators} on the hypercyclicity of operators of the form $T_{\bar d,w}$ on $\ell^{2}(\mathbb N)$. Indeed, condition~\eqref{Eq_Grivaux Sufficient Condition} directly implies condition~\eqref{Eq_Cor Perturbed Weighted Shift Strong Sufficient Condition}. Moreover, the conclusion of Corollary~\ref{Cor_Perturbed Weighted Shift Sufficient Condition} holds on $\ell^{p}\lc\N\rc$ for every $1\le p<\infty$.

\paragraph{}

The second main result provides an analytic characterization of the shadowing property for a perturbed weighted shift $T_{\bar d,\bar w}$ on $\ell^{p}(\mathbb N)$ in terms of the defining sequences $\bar d$ and $\bar w$.

\begin{theorem}\label{Theor_Shadowing Petrurbed Weighted Shifts}
Let $T_{\bar d,\bar w}:\ell^{p}(\mathbb N)\to\ell^{p}(\mathbb N)$, with $1\le p<+\infty$, be a perturbed weighted shift. Then $T_{\bar d,\bar w}$ has the shadowing property if and only if
\begin{equation}\label{Eq_Theor Analytic Characterization of Shadowing}
1\notin
\bigl[r_{inf}(\lambda),r_{sup}(\lambda)\bigr]
\quad\text{for every }\lambda\in\mathbb S^{1},
\end{equation}
where
\begin{equation}\label{Eq_Shadowing Inf Bound} 
		r_{inf}\lc \lambda \rc = \lim_{l \to \infty}\inf_{k \ge 0} \lav \dfrac{\lc d_{k+1}-\lambda  \rc \cdots \lc d_{k+l}-\lambda \rc}{w_{k+1} \cdots w_{k+l}}   \rav^{\frac{1}{l}}   
	\end{equation}
	and    
	\begin{equation}\label{Eq_Shadowing Sup Bound} 
		r_{sup}(\lambda)=\lim_{l \to \infty}\sup_{k\ge 0} \lav \dfrac{\lc d_{k+1}-\lambda  \rc \cdots \lc d_{k+l}-\lambda \rc} {w_{k+1} \cdots w_{k+l}}\rav^{\frac{1}{l}}. 
	\end{equation} 
    
\end{theorem}

\begin{remark}
If $d_k=d$ for every $k\ge1$, then $T_{\bar d,\bar w}$ has the
shadowing property if and only if
$$
\lim_{l\to\infty}\sup_{k\ge0}
\lav w_{k+1}\cdots w_{k+l}\rav^{\frac1l}
<
\lav\lav d\rav-1\rav
\quad\text{or}\quad
\lim_{l\to\infty}\inf_{k\ge0}
\lav w_{k+1}\cdots w_{k+l}\rav^{\frac1l}
>
\lav d\rav+1.
$$
More generally, if $(d_k)_{k\ge1}$ is periodic with period $m$, for each fixed $\lambda\in\mathbb S^1$,
$\lav\prod_{j=1}^{l}(d_{k+j}-\lambda)\rav^{1/l}$
converges uniformly in $k\ge0$, as $l$ tends to infinity, to
$\lav\prod_{j=1}^{m}(d_j-\lambda)\rav^{1/m}$. Therefore
$T_{\bar d,\bar w}$ has the shadowing property if and only if
$$
\max_{\lambda\in\mathbb S^1}
\lav\lc d_1-\lambda\rc\cdots\lc d_m-\lambda\rc\rav^{\frac1m}
<
\lim_{l\to\infty}\inf_{k\ge0}
\lav w_{k+1}\cdots w_{k+l}\rav^{\frac1l},
$$
or
$$
\min_{\lambda\in\mathbb S^1}
\lav\lc d_1-\lambda\rc\cdots\lc d_m-\lambda\rc\rav^{\frac1m}
>
\lim_{l\to\infty}\sup_{k\ge0}
\lav w_{k+1}\cdots w_{k+l}\rav^{\frac1l}.
$$
\end{remark}

Theorem~\ref{Theor_Shadowing Petrurbed Weighted Shifts} recovers, in particular, the necessary and sufficient condition~\eqref{Eq_Weighted Shifts Shadowing Characterization} for the shadowing property of weighted shifts established by Bernardes and Messaoudi \cite{Bernardes_Messaoudi-Shadowing_and_Structural_Stability}.

\begin{remark}\label{Rem_S1-Complemented}   
An operator $T:\mc{B}\to\mc{B}$ is said to be \textit{$\SP^{1}$-complemented} if
\begin{equation}\label{Eq_Closed Complement}
\ker\lc T-\lambda I\rc
\text{ has a closed complement for every }
\lambda\in\SP^{1},
\end{equation}
that is, for every $\lambda\in\SP^{1}$, there exists a closed subspace $\mc M_{\lambda}\sub\mc B$ such that
$\mc B=\ker\lc T-\lambda I\rc\oplus\mc M_{\lambda}$.
The class of $\SP^{1}$-complemented operators includes all operators on Hilbert spaces, as well as perturbed weighted shifts $T_{\bar d,\bar w}$, as will be proved for the latter at the beginning of Section~\ref{section 5}. Moreover, \cite[Theorem~3.5]{Messaoudi_Neto_Saavedra_Tsokanos-On_GH_S_SH_for_Linear_Operators} guarantees that the shadowing property, topological stability, and strong Lipschitz structural stability are equivalent for invertible $\SP^{1}$-complemented operators. Consequently, in view of Theorem~\ref{Theor_Shadowing Petrurbed Weighted Shifts}, an invertible perturbed weighted shift $T_{\bar d,\bar w}:\ell^{p}\lc\N\rc\to\ell^{p}\lc\N\rc$ with $1\le p<+\oo$ is strongly Lipschitz structurally stable (respectively, topologically stable) if and only if condition~\eqref{Eq_Theor Analytic Characterization of Shadowing} holds.
\end{remark}

The next corollary follows directly from Corollary~\ref{Cor_Perturbed Weighted Shift Sufficient Condition} and Theorem~\ref{Theor_Shadowing Petrurbed Weighted Shifts} and shows that there is a large family of perturbed weighted shifts on $\ell^{p}\lc\N\rc$, $1\le p<+\infty$, that are both frequently hypercyclic and have the shadowing property.

\begin{corollary}
Let $T_{\bar d,\bar w}:\ell^{p}(\mathbb N) \to \ell^{p}(\mathbb N)$, $1 \le p < +\infty$, be a perturbed weighted shift satisfying
\begin{equation*}
\max\lfp
2\sup_{k\ge1}\lav d_k\rav,
1+\sup_{k\ge1}\lav d_k\rav
\rfp
<
\inf_{k\ge1}\lav w_k\rav.
\end{equation*}
Then $T_{\bar d,\bar w}$ is frequently hypercyclic and has the shadowing property.
\end{corollary}

\section{Proof of Theorem~\ref{Theor_H and FH for Perturbed Weighted Shifts} and Corollary~\ref{Cor_Perturbed Weighted Shift Sufficient Condition}}
\label{Sec_Proof of Hypercyclicity}

The proof of Theorem~\ref{Theor_H and FH for Perturbed Weighted Shifts} proceeds by applying Theorem \ref{Theor_SHFH Criterion}. To this end, throughout this section, fix $1 \le p < +\oo$, and let
$T_{\bar d,\bar w}:\ell^{p}(\mathbb N)\to\ell^{p}(\mathbb N)$
be a perturbed weighted shift. Define
\begin{equation}\label{Eq_Dense Subset Definition}
D_0
=
\bigcup_{m=1}^{\infty}
\ker T_{\bar d,\bar w}^{\,m},
\end{equation}
that is, the union of the kernels of the iterates of $T_{\bar d,\bar w}$. 
Also, define the linear map
$S_{\bar d,\bar w}: \ell^{p}\lc \N \rc \to \C^{\N}$
whose $k$-th coordinate is given by
\begin{equation}\label{Eq_Inverse Operator Definition}
\bigl(S_{\bar d,\bar w} \lc x_{l} \rc_{l =1}^{+\oo} \bigr)_k
=
\begin{cases}
0,
& \text{if }k=1,\\[1ex]
\displaystyle
\frac{x_{k-1}}{w_{k-1}}
-
\sum_{l=1}^{k-2}
(-1)^{k-l}
x_l
\frac{d_{l+1}\cdots d_{k-1}}
{w_l\cdots w_{k-1}},
& \text{if }k\ge2.
\end{cases}
\end{equation}
The sum appearing in each coordinate is finite; thus, $S_{\bar d,\bar w}$ is well-defined as a map into $\C^{\N}$.

The following lemma shows that, under assumption~\eqref{Eq_Bounded Inverse Vector}, the map $S_{\bar d,\bar w}$ leaves $D_0$ invariant and that $T_{\bar d,\bar w}\circ S_{\bar d,\bar w}$ is the identity on $D_0$. 
It also identifies $D_0$ with the linear span of the vectors $\bm v_m$ defined in \eqref{Eq_Kernel Basis Vectors}.

\begin{lemma}\label{Lemma_Properties of the Inverse Operator}
Let $S_{\bar{d}, \bar{w}} : \ell^{p}\lc \N \rc \to \C^{\N} $ be the linear operator defined in \eqref{Eq_Inverse Operator Definition}, and let $D_{0}$ be the set defined in \eqref{Eq_Dense Subset Definition}. Assume that condition \eqref{Eq_Bounded Inverse Vector} holds. Then the following assertions hold:

\begin{enumerate}
	
	\item For every $m \ge 1$,
	\begin{equation}\label{Eq_Inverse Vectors and Inverse Operator}
		\bm{v}_{m+1}=S_{\bar{d},\bar{w}}^{\,m}(\bm{v}_{1}) \in \ell^{p}\lc \N \rc,
	\end{equation}
	where the vectors $\bm{v}_{m}$ are defined in \eqref{Eq_Kernel Basis Vectors}.
	
	\item The set $D_{0}$ satisfies 
	\begin{equation}\label{Eq_Spanned Vector Space and Kernels}
		D_{0}
		=
		\bigcup_{m\ge1}\ker T_{\bar{d},\bar{w}}^{\,m}
		=
		\operatorname{span}\langle \bm{v}_{m} : \; m \ge 1 \rangle.
	\end{equation}
    In particular, $\ker T_{\bar{d}, \bar{w}}^{m} = \operatorname{span}\langle \bm{v}_{1}, \dots, \bm{v}_{m} \rangle$  and $ S_{\bar{d},\bar{w}}(D_{0})\subseteq D_{0}  \sub \ell^{p}\lc \N \rc$. 
	
	\item For every $\bm{x} \in D_{0}$,
	\begin{equation}\label{Eq_Inverse Operator Equality}
		T_{\bar{d},\bar{w}}\circ S_{\bar{d},\bar{w}}\lc \bm{x}  \rc
		=
        \bm{x} .
	\end{equation}

\end{enumerate}

\end{lemma}

\begin{proof}

In view of \eqref{Eq_Kernel Basis Vectors}, condition \eqref{Eq_Bounded Inverse Vector} ensures that $\bm{v}_{m}\in\ell^{p}\lc \N \rc$ for every $m\ge1$. It is first shown that $S_{\bar{d},\bar{w}}(\bm{v}_{m})=\bm{v}_{m+1}$ for every $m\ge1$. For $k \ge 1$, the $k$-th coordinate of $\bm{v}_{m}$ is $
(\bm v_m)_k
= D_{k-1}\lc k-m\rc / W_{k-1}$.  
By the recurrence relation $D_r(s)=D_{r-1}(s)-d_rD_{r-1}(s-1)$, one has 
$$D_{k-2}(k-2-m)= \sum_{l=m}^{k-2}(-1)^{k-l}D_{l-1}(l-m)d_{l+1}\cdots d_{k-2}.$$
Consequently, 
\begin{align*}
    (S_{\bar{d},\bar{w}}(\bm v_m))_{k} =  \dfrac{ D_{k-2}\lc k-1-m \rc-d_{k-1}D_{k-2}(k-2-m) } {W_{k-1}} =  \dfrac{D_{k-1}(k-(m+1))}{W_{k-1}} = \lc \bm v_{m+1} \rc_{k}  .
\end{align*}
Therefore, $ S_{\bar{d},\bar{w}}(\bm{v}_{m})=\bm{v}_{m+1}$, $m \ge 1$, and hence \eqref{Eq_Inverse Vectors and Inverse Operator} follows.

\paragraph{}

To establish \eqref{Eq_Spanned Vector Space and Kernels}, it suffices to prove by induction that
$\ker T_{\bar d,\bar w}^{\,m}
=
\operatorname{span}\{\bm v_1,\ldots,\bm v_m\}$ for every $m\ge1$.
Since $\bm v_1$ spans $\ker T_{\bar d,\bar w}$, the assertion holds for $m=1$. 
Assume that it holds for some $m\ge1$ and let
$\bm{x} \in\ker T_{\bar d,\bar w}^{\,m+1}$. Then
$T_{\bar d,\bar w}(\bm{x})\in\ker T_{\bar d,\bar w}^{\,m}$, so, by the induction hypothesis, 
$$
T_{\bar d,\bar w}(\bm{x})
=
\sum_{j=1}^{m}a_j\bm v_j
=
T_{\bar d,\bar w}
\left(
\sum_{j=1}^{m}a_j\bm v_{j+1}
\right)
$$
for some scalars $a_1,\ldots,a_m$, where the last equality follows from \eqref{Eq_Inverse Vectors and Inverse Operator}. Hence,
$
\bm{x} -\sum_{j=1}^{m}a_j\bm v_{j+1}
\in\ker T_{\bar d,\bar w}
=
\operatorname{span}\{\bm v_1\}$, 
and therefore
$\bm{x} \in\operatorname{span}\{\bm v_1,\ldots,\bm v_{m+1}\}$.
This proves
$$
\ker T_{\bar d,\bar w}^{\,m+1}
\subseteq
\operatorname{span}\{\bm v_1,\ldots,\bm v_{m+1}\}.$$ 
Conversely, \eqref{Eq_Inverse Vectors and Inverse Operator} gives
$T_{\bar d,\bar w}^{\,j-1}(\bm v_j)=\bm v_1$ and
$T_{\bar d,\bar w}(\bm v_1)=0$ for every $1\le j\le m+1$. Thus,
$\bm v_j\in\ker T_{\bar d,\bar w}^{\,m+1}$ for every $1\le j\le m+1$, which gives the reverse inclusion and completes the induction. 
Consequently, $D_0
=
\bigcup_{m\ge1}\ker T_{\bar d,\bar w}^{\,m}
=
\operatorname{span}\langle \bm v_m:m\ge1 \rangle,$ as claimed.

\paragraph{}

It remains to prove \eqref{Eq_Inverse Operator Equality}. Let $ \bm{x} \in D_{0}$.
For $k=1$,
$$
\bigl(T_{\bar{d},\bar{w}}\circ S_{\bar{d},\bar{w}}\lc \bm{x}  \rc \bigr)_{1}
=
d_{1}\bigl(S_{\bar{d},\bar{w}}\lc \bm{x}  \rc \bigr)_{1}
+
w_{1}\bigl(S_{\bar{d},\bar{w}}\lc \bm{x}  \rc \bigr)_{2}
=
w_{1}\frac{x_{1}}{w_{1}}
=
x_{1}.
$$
Now let $k\ge2$. By the definitions of $T_{\bar{d}, \bar{w}}$ and $S_{\bar{d}, \bar{w}}$, 
\begin{align*}
\bigl(T_{\bar{d},\bar{w}}\circ S_{\bar{d},\bar{w}}\lc \bm{x}  \rc \bigr)_{k}
&=
d_{k}\bigl(S_{\bar{d},\bar{w}}\lc \bm{x}  \rc \bigr)_{k}
+
w_{k}\bigl(S_{\bar{d},\bar{w}}\lc \bm{x}  \rc \bigr)_{k+1}
\\ 
&=
d_{k}\frac{x_{k-1}}{w_{k-1}} +
\sum_{l=1}^{k-2}
(-1)^{l}
\left(
\prod_{j=0}^{l}
\frac{d_{k-j}}{w_{k-1-j}}
\right) x_{k-1-l}
\\
&\quad
+ x_{k}
- w_{k}\frac{d_{k}}{w_{k}}
\frac{x_{k-1}}{w_{k-1}}
+
w_{k}
\sum_{l=2}^{k-1}
(-1)^{l}
\left(
\prod_{j=1}^{l}
\frac{d_{k+1-j}}{w_{k+1-j}}
\right)
\frac{x_{k-l}}{w_{k-l}}
=
x_{k}. 
\end{align*} 
This proves \eqref{Eq_Inverse Operator Equality} and completes the proof. 
\end{proof}

\paragraph{}

To apply Theorem \ref{Theor_SHFH Criterion}, it remains to prove that
$D_{0}$ is dense in $\ell^{p} \lc \N \rc$.  By \eqref{Eq_Spanned Vector Space and Kernels}, this is equivalent to showing that
\begin{equation}\label{Eq_Claim Dense Kernels}
\overline{\operatorname{span}\langle \bm v_m : m\ge1\rangle}
\underset{\eqref{Eq_Spanned Vector Space and Kernels}}{=}
\ell^{p} \lc \N \rc.
\end{equation}
The following lemma establishes the required density.

\begin{lemma}\label{Lem_Density Kernel Vectors}
Assume that condition~\eqref{Eq_Dense Kernels Sufficient Condition} holds. Then $\langle \bm v_m : m\ge1\rangle$ is dense in $\ell^{p}(\mathbb N)$.
\end{lemma}

 \begin{proof}
 
Since the linear span of the canonical basis $\{\bm e_n\}_{n\ge1}$ is dense in $\ell^{p}\lc \N \rc$, it suffices to show that every basis vector $\bm{e}_{n}$ belongs to
$\overline{\operatorname{span}\langle \bm v_m : m\ge1\rangle}$.
Here, $\bm e_n=(\delta_{n,k})_{k=1}^{\infty}
=( \underbrace{0,\ldots,0}_{n-1}, 1,0,\ldots),$ where $\delta_{n,k}$ denotes the Kronecker delta.
To this end, it suffices to show that for every $n\ge1$,
\begin{equation}\label{Eq_Orthonormal Basis Expantion}
\bm e_n
=
W_{n-1}
\lim_{M\to\infty}
\sum_{m=1}^{M}
Q_n(m-n)\,\bm v_m
\qquad\text{in }\ell^{p} \lc \N \rc.
\end{equation} 

Notice that, by \eqref{Eq_Kernel Basis Vectors}, condition \eqref{Eq_Dense Kernels Sufficient Condition} guarantees that the series appearing on the right-hand side of \eqref{Eq_Orthonormal Basis Expantion} converges in $\ell^p(\N)$. It therefore remains to verify the coordinatewise identity
	\begin{equation}\label{Eq_Kernel Combinatorics}
		\delta_{n,k}
		=
		\frac{W_{n-1}}{W_{k-1}}
		\sum_{m=n}^{k}
		Q_n(m-n)\,
		D_{k-1}(k-m),
		\qquad n,k\ge1.
	\end{equation}
If $1\le k<n$, then the sum is empty, and hence both sides  of \eqref{Eq_Kernel Combinatorics} are equal to zero.
If $k=n$, then 
$$ \delta_{n,n}
	=
	\frac{W_{n-1}}{W_{n-1}}
	Q_n(0)D_{n-1}(0)
	=
	1 ,$$
since $Q_n(0)=D_{n-1}(0)=1$. 
Finally, consider the case $k>n$, and set $r=k-n$. Recall that the generating functions of $Q_n$ and $D_{k-1}$ are given in \eqref{Eq_Generating Functions of D and Q}. Thus, the sum on the right-hand side of \eqref{Eq_Kernel Combinatorics} is the $r$-th coefficient of the formal series 
\begin{align*}
\left(
\sum_{j=0}^{\infty}Q_n(j)t^j
\right)
\left(
\sum_{j=0}^{k-1}D_{k-1}(j)t^j
\right)
&\underset{\eqref{Eq_Generating Functions of D and Q}}{=}
\prod_{j=n+1}^{k-1}(1-d_{j} t).
\end{align*}
Since the polynomial on the right-hand side has degree at most $r-1$, its $r$-th coefficient is zero. Consequently, $
\delta_{n,k}=0$, as required. 

This completes the proof of the lemma.
\end{proof}

\begin{proof}[Proof of Theorem \ref{Theor_H and FH for Perturbed Weighted Shifts}]

Assume that conditions \eqref{Eq_Bounded Inverse Vector} and \eqref{Eq_Dense Kernels Sufficient Condition} hold. Then, by Lemma~\ref{Lemma_Properties of the Inverse Operator}, the subset $D_{0} =\bigcup_{m=1}^{\infty}\ker T_{\bar d,\bar w}^{m} $ is dense in $\ell^{p}\lc \N \rc$ and $ T_{\bar d,\bar w}\circ S_{\bar d,\bar w}x=x$, for all $ x\in D_{0} $.

The proof proceeds by applying Theorem \ref{Theor_SHFH Criterion} to the operator
$T_{\bar d,\bar w}$, the map $S_{\bar d,\bar w}$, and the subset $
D_{0}$.
It remains only to verify the hypotheses corresponding to the three assertions of
Theorem~\ref{Theor_SHFH Criterion}.

\paragraph{Point (1): Supercyclicity.}

By the definition of $D_{0}$, for every $x\in D_{0}$ there exists an integer $m_{x}\ge1$ such that
\begin{equation}\label{Eq_T Finally Zero}
T_{\bar d,\bar w}^{m}x=0,
\qquad
m\ge m_{x}.
\end{equation}
Consequently, for every $x,y\in D_{0}$, $ \|T_{\bar d, \bar w}^{m}x\|_{p} \,\|S_{\bar d, \bar w}^{m}y\|_{p} =0 $ whenever $m\ge m_{x}$.
Condition \eqref{Eq_S Criterion Limit Zero} therefore holds, and
$T_{\bar d,\bar w}$ is supercyclic.

\paragraph{Point (2): Hypercyclicity.}

Assume first that condition
\eqref{Eq_Hypercyclic Sufficient Condition Liminf} holds. By relation \eqref{Eq_Kernel Basis Vectors}, one has that $\ldav \bm{v}_{m} \rdav_{p}^{p} = \sum_{k=m}^{+\oo} \lc \lav D_{k-1}\lc k -m \rc \rav^{p}/ \lav W_{k-1} \rav^{p} \rc $ and, therefore, for all $n\ge 0$, 
$$ \liminf_{m\to \infty} \lc \|\bm v_{m}\|_{p} +\cdots+\|\bm v_{m+n}\|_{p} \rc 
\underset{\eqref{Eq_Inverse Vectors and Inverse Operator},\eqref{Eq_Inverse Operator Equality}}{\le}
\lc \sum_{j=0}^n\lav\lav T_{\bar d,\bar w}^j\rav \rav \rc\liminf_{m\to \infty}||\bm v_m||_{p} \underset{\eqref{Eq_Hypercyclic Sufficient Condition Liminf}}{=}0. $$ 
Consequently, for every integer $n \ge0$, there exists an integer $m_{n}$ such that
$$
\|\bm v_{m_{n}}\|_{p} + \cdots + \|\bm v_{m_{n}+n}\|_{p}
<
\frac1{n+1}.
$$ 

Fix $ \bm{x} =\sum_{j=1}^{N}a_{j}\bm v_{j}\in D_{0} $. By relation \eqref{Eq_Inverse Vectors and Inverse Operator}, one infers that $ 
S_{\bar d,\bar w}^{m_{n}}\lc \bm{x}  \rc
=
\sum_{j=1}^{N}
a_{j}\bm v_{m_{n}+j}$. Therefore, for all $n \ge N$,
$$
\|S_{\bar d,\bar w}^{m_{n}} \lc \bm{x}  \rc\|
\le
\sum_{j=1}^{N}|a_{j}|
\,
\|\bm v_{m_{n}+j}\|_{p} \le   \dfrac{\max_{1\le j \le N}  \lav a_{j} \rav}{n+1}  .
$$
Upon taking the limit $n \to +\oo$, it follows that $ S_{\bar d,\bar w}^{m_{n}}x\longrightarrow 0$. 
Thus, the last limit, taken together with \eqref{Eq_T Finally Zero}, yields that condition
\eqref{Eq_H Criterion LimInf Condition}
is satisfied.
Therefore $T_{\bar d,\bar w}$ is hypercyclic.

\paragraph{}

Conversely, assume that
$\{ W_{m-1} \bm v_{m}\}_{m\ge1}$ is a Schauder basis and that
$T_{\bar d,\bar w}$ is hypercyclic. Equivalently, by relation \eqref{Eq_Inverse Vectors and Inverse Operator} of Lemma \ref{Lemma_Properties of the Inverse Operator}, one has that $\{ W_{m-1} \; S_{\bar{d}, \bar{w}}^{m-1}\bm v_{1}\}_{m\ge1}$
is a Schauder basis of $\ell^{p}(\mathbb N)$. 
Let $ x_{h}
=
\sum_{m=1}^{+\infty}
a_{m} W_{m-1}
S_{\bar{d}, \bar{w}}^{m-1}\bm v_{1} $ be a hypercyclic vector.
Consequently, there exists a sequence
$(m_{n})$ such that $\lim_{n\to +\infty}T_{\bar d,\bar w}^{m_{n}}x_{h}= e_{1}$. Since 
$$
T^{m_n}x_{h}
=
a_{m_n + 1} W_{m_{n}}  \bm v_{1}
+
\sum^{+\infty}_{k=2}
a_{m_n+k} {W_{m_{n} + k-1} \over W_{k-1}} W_{k-1} \bm v_{k} , 
$$
the uniqueness of the Schauder expansion implies that $
\lim_{n \to +\oo} a_{m_{n}+1} W_{m_{n}}= 1.$ Moreover, every convergent series in a Banach space has terms tending to zero, hence 
$$\underset{n \to +\infty}{\lim} a_{m_{n}+1} W_{m_{n}}\bm v_{m_{n}+1} = 0.$$

Consequently $\underset{n \to +\infty}{\lim} \|\bm v_{m_{n}+1}\|_{p}
=0,$ This proves
\eqref{Eq_Hypercyclic Sufficient Condition Liminf}.

\paragraph{Point (3): Frequent hypercyclicity.}

For every $ \bm{x} \in D_{0}$,
the series $ \sum_{m=0}^{+\infty}T^{m}x $ 
contains only finitely many non-zero terms by
\eqref{Eq_T Finally Zero}; hence it converges unconditionally. 
It therefore remains to prove the unconditional convergence of
$ \sum_{m=0}^{+\infty}S^{m}\lc \bm{x}  \rc $. 
Since $D_{0}$ is the linear span of
$\{\bm v_{m}\}_{m\ge1}$,
it is enough to consider vectors of the form
$ \bm{x} =\bm v_{n}$.
By relation \eqref{Eq_Inverse Vectors and Inverse Operator}, $ S^{m}\bm v_{n}
=
\bm v_{n+m}$ and, therefore, $ \sum_{m=0}^{+\infty}S^{m}\bm v_{n}
=
\sum_{m=n}^{+\infty}\bm v_{m}$, 
which converges unconditionally due to
\eqref{Eq_Freq Hypercyclic Sufficient Condition Series Convergence} and \eqref{Eq_Kernel Basis Vectors}. Thus condition
\eqref{Eq_FH Criterion Convergent Series Condition}
is satisfied, and
$T_{\bar d,\bar w}$ is frequently hypercyclic.

To show that $\SP^{1} \sub \sigma_{pt}\lc T_{\bar{d}, \bar{w}} \rc $, fix $\lambda \in \SP^{1}$.  Since the series $\sum_{m=1}^{+\oo} \bm{v}_{m}$ 
converges unconditionally by assumption \eqref{Eq_Freq Hypercyclic Sufficient Condition Series Convergence},
the Bounded Multiplier Test \cite[Theorem 4.2.8]{Megginson-An_Introduction_to_Banach_Space_Theory} guarantees that the vector 
$$ \bm{w}_{\lambda} = \sum_{m=1}^{+\oo} \lambda^{m-1} \, \lc \dfrac{ D_{k-1}\lc k-m \rc } {W_{k-1}}  \rc_{k=1}^{+\oo}  $$
is well defined in $\ell^{p}\lc \N \rc$. Since $T_{\bar d,\bar w}(\mathbf v_{m+1})=\mathbf v_{m}$,  for all $ m\ge1$, one infers that  $ T\lc \bm{w}_{\lambda} \rc = \lambda\bm{w}_{\lambda} $. In other words, $\lambda \in \sigma_{pt}\lc T_{\bar{d}, \bar{w}} \rc$. 

\smallskip

Finally, to establish that condition~\eqref{Eq_Strong Condition for FH} implies \eqref{Eq_Freq Hypercyclic Sufficient Condition Series Convergence}, consider an arbitrary subsequence $(\bm v_{m_l})_{l\ge1}$ of $(\bm v_m)_{m\ge1}$. Since the subsequence is arbitrary, it suffices to show that the subseries $\sum_{l=1}^{\infty}\bm v_{m_l}$ converges in $\ell^p(\mathbb N)$. To this end, for every $N\ge1$,
\begin{align*}
\ldav\sum_{l=1}^{N}\bm v_{m_l}\rdav_p
\underset{\eqref{Eq_Kernel Basis Vectors}}{\le}
\ldav
\lc
\dfrac{\sum_{l=1}^{N}\lav D_{k-1}\lc k-m_l\rc\rav}
{\lav W_{k-1}\rav}
\rc_{k=1}^{+\oo}
\rdav_p \underset{\eqref{Eq_D Quantities Definition}}{\le}
\ldav
\lc
\dfrac{\prod_{j=1}^{k-1}(1+\lav d_j\rav)}
{\lav w_1\cdots w_{k-1}\rav}
\rc_{k=1}^{+\oo}
\rdav_p
\underset{\eqref{Eq_Strong Condition for FH}}{<}
+\infty.
\end{align*}
Moreover, for every fixed $k\ge1$, only finitely many terms $D_{k-1}\lc k-m_l\rc$ are nonzero. Hence, the partial sums $\sum_{l=1}^{N}\bm v_{m_l}$ converge coordinatewise as $N\to+\infty$. Since the dominating sequence $\bar z=\lc z_k\rc_{k\ge1}$, where $z_k= \prod_{j=1}^{k-1} \lc 1+\lav d_j\rav \rc / \lav w_{1}\cdots w_{k-1}\rav $, belongs to $\ell^p(\mathbb N)$ by \eqref{Eq_Strong Condition for FH}, the dominated convergence theorem yields the convergence of $\sum_{l=1}^{\infty}\bm v_{m_l}$ in $\ell^p(\mathbb N)$. Therefore, $\sum_{m=1}^{\infty}\bm v_m$ converges unconditionally.

The proof is complete. 
\end{proof}

\paragraph{} 
Finally, the proof of Corollary \ref{Cor_Perturbed Weighted Shift Sufficient Condition} will make use of the following lemma.

\begin{lemma}\label{Lem_Norm of S}
Let $c_{\bar{d}, \bar{w}} $ be the quantity defined in \eqref{Eq_Cor Inverse Operator Norm Bound}, and let $S_{\bar{d}, \bar{w}} : \ell^{p}\lc \N \rc \to \C^{\N}$ be the linear map defined in \eqref{Eq_Inverse Operator Definition}. 
Assume that $c_{\bar{d}, \bar{w}} < +\oo$. Then $S_{\bar{d}, \bar{w}} \lc \ell^{p}\lc \N \rc \rc  \sub \ell^{p}\lc \N \rc$. In particular, the linear operator $S_{\bar{d}, \bar{w}} : \ell^{p}\lc \N \rc \to \ell^{p}\lc \N \rc$ is well defined and $\ldav S_{\bar{d}, \bar{w}} \rdav \le c_{\bar{d}, \bar{w}}$. 

Moreover, if $ 2 \sup_{k \ge 1}  \lav d_{k} \rav < \inf_{k \ge 1}  \lav w_{k} \rav$, then $c_{\bar{d}, \bar{w}} < 1 / \ldav I_{\bar{d}} \rdav$. Furthermore, if $c_{\bar{d}, \bar{w}} < 1 / \ldav I_{\bar{d}} \rdav$, then  condition \eqref{Eq_Dense Kernels Sufficient Condition} is satisfied.

\end{lemma}

\begin{proof}
Assume that $c_{\bar{d}, \bar{w}} < +\oo$. Fix $\lc x_{k} \rc_{k =1}^{+\oo} \in \ell^{p}\lc \N \rc$. To prove that $S_{\bar{d}, \bar{w}} \lc \ell^{p}\lc \N \rc \rc  \sub \ell^{p}\lc \N \rc$, it suffices to show that $\ldav S_{\bar{d}, \bar{w}} \lc x_{k} \rc_{k=1}^{+\oo} \rdav_{p} \le c_{\bar{d}, \bar{w}} \ldav \lc x_{k} \rc_{k=1}^{+\oo} \rdav_{p}$. To this end, the last inequality follows by an immediate application of triangle inequality: 
$$  
\lav\lav S_{\bar d,\bar w }(x_k)_{k=1}^\infty\rav\rav_{p} 
\underset{ \eqref{Eq_Inverse Operator Definition}}{\le} 
\ldav \lc \frac{x_{k}}{w_{k}}\rc_{k=1}^{+\oo}\rdav_{p} + \sum_{l=1}^{+\oo}\ldav \lc \underbrace{0,\ldots,0}_{l} , \bm{y}\lc l \rc \rc \rdav_{p} \underset{ \eqref{Eq_Cor Inverse Operator Norm Bound}}{\le}  c_{\bar d,\bar w}|| \lc x_{k} \rc_{k=1}^{+\infty}|| ,  $$
where, for each $l \ge 1$,  $\bm{y}\lc l \rc = \lc  \prod_{j=1}^{l} \lc d_{k+j} / w_{k+j }\rc \; x_{k} / w_{k} 
\rc_{k=1}^{+\oo}$.

Finally, on the one hand, if  $ 2 \sup_{k \ge 1}  \lav d_{k} \rav < \inf_{k \ge 1}  \lav w_{k} \rav$, then it is readily verified that $c_{\bar{d}, \bar{w}} < 1/ \ldav I_{\bar{d}} \rdav$. On the other hand, if $c_{\bar{d}, \bar{w}} < 1/ \ldav I_{\bar{d}} \rdav$, then 
$$  
\sum_{m=1}^{\infty} \lav Q_n(m-n)\rav \ldav \bm{v}_{m} \rdav_{p}  \le ||\bm v_1||c_{\bar{d},\bar{w}}^{n-1}\sum_{m=n}^{+\infty}\dfrac{m^{n-1}}{(n-1)!}\lc c_{\bar d, \bar w}||I_d||\rc^{(m-n)}<+\infty, $$
where the last sum is finite since the series decays exponentially. Therefore, the series in \eqref{Eq_Dense Kernels Sufficient Condition} converges absolutely and thus unconditionally. This completes the proof.     
\end{proof}

\begin{proof}[Proof of Corollary \ref{Cor_Perturbed Weighted Shift Sufficient Condition}]

Let $c_{\bar{d}, \bar{w}} $ be the quantity defined in \eqref{Eq_Cor Inverse Operator Norm Bound}. Assume that condition \eqref{Eq_Cor Perturbed Weighted Shift Sufficient Condition} holds. In view of Theorem \ref{Theor_H and FH for Perturbed Weighted Shifts}, to show that the operator $T_{\bar{d}, \bar{w}}$ is frequently hypercyclic, it suffices to show that the conditions \eqref{Eq_Bounded Inverse Vector}, \eqref{Eq_Dense Kernels Sufficient Condition}, and \eqref{Eq_Freq Hypercyclic Sufficient Condition Series Convergence} are satisfied. 
To this end, since $ \ldav S_{\bar{d}, \bar{w}} \rdav \le c_{\bar d, \bar w}< 1/ \ldav I_{\bar d} \rdav$, Lemma \ref{Lem_Norm of S} guarantees that condition \eqref{Eq_Dense Kernels Sufficient Condition} holds. 
To verify \eqref{Eq_Bounded Inverse Vector} and \eqref{Eq_Freq Hypercyclic Sufficient Condition Series Convergence}, notice that \eqref{Eq_Cor Perturbed Weighted Shift Sufficient Condition} immediately implies $\ldav \bm{v}_{1} \rdav_{p} < +\oo$. In turn, since $\ldav S_{\bar{d}, \bar{w}} \rdav_{p} \le c_{\bar{d}, \bar{w}} < 1$, one has 
$$ \sum_{m=1}^{+\oo} \ldav \bm{v}_{m} \rdav_{p}  \underset{ \eqref{Eq_Cor Perturbed Weighted Shift Sufficient Condition}, \eqref{Eq_Inverse Vectors and Inverse Operator} }{\le}  { 1 \over 1 - c_{\bar{d}, \bar{w}} } \ldav \bm{v}_{1} \rdav_{p} < +\oo . $$
In other words, the series of the vectors $\bm{v}_{m}$, $m \ge 1$, converges absolutely and, therefore, both conditions \eqref{Eq_Bounded Inverse Vector} and \eqref{Eq_Freq Hypercyclic Sufficient Condition Series Convergence} are satisfied.

\smallskip 

Finally, it remains to establish the special case that condition~\eqref{Eq_Cor Perturbed Weighted Shift Strong Sufficient Condition} implies \eqref{Eq_Cor Perturbed Weighted Shift Sufficient Condition}. Since
$2\sup_{k\ge1}\lav d_k\rav<\inf_{k\ge1} \lav w_k\rav$, Lemma~\ref{Lem_Norm of S} yields
$c_{\bar d,\bar w}<1/\ldav I_{\bar d}\rdav$. 
Moreover, since
$1+\sup_{k\ge1}\lav d_k\rav<\inf_{k\ge1}\lav w_k\rav$, it follows that
$$
c_{\bar d,\bar w}
\underset{ \eqref{Eq_Cor Inverse Operator Norm Bound} }{\le} \dfrac{1}{ \inf_{k \ge 1} \lav w_{k} \rav }\lc 1+  \sum_{l=1}^{\infty} \lc { \sup_{k\ge 1} \lav d_{k} \rav \over 1 + \sup_{k \ge 1}\lav d_{k} \rav} \rc^{l} \rc  
\underset{ }{\le}  1 .
$$

The proof is complete. 
\end{proof}

\section{Proof of Theorem~\ref{Theor_Shadowing Petrurbed Weighted Shifts}}\label{section 5}

First, the claim made in Remark~\ref{Rem_S1-Complemented} is established: every perturbed weighted shift $T_{\bar d,\bar w}$ on $\ell^{p}\lc\N\rc$ is $\SP^{1}$-complemented. To verify condition~\eqref{Eq_Closed Complement}, it suffices to note that every closed finite-dimensional subspace of a Banach space admits a closed complement and that, for every $\lambda\in\C$, the estimate
$\dim\ker(T_{\bar d,\bar w}-\lambda I)\le1$
holds.

\paragraph{}

To prove Theorem~\ref{Theor_Shadowing Petrurbed Weighted Shifts}, it suffices to show that
\begin{equation}\label{Eq_Shadowing Spectral-Analytic Conditions}
\text{for all } \lambda \in \SP^{1}, \quad
\lambda \not\in \sigma_{sur}\lc T_{\bar{d}, \bar{w}}\rc
\quad\text{if and only if}\quad
1\not\in\lb r_{inf}\lc\lambda\rc,r_{sup}\lc\lambda\rc\rb,
\end{equation}
where the quantities $r_{inf}\lc\lambda\rc$ and $r_{sup}\lc\lambda\rc$ are defined in \eqref{Eq_Shadowing Inf Bound} and \eqref{Eq_Shadowing Sup Bound}, respectively.
Indeed, a recent result of Dragičević and Pituk \cite[Theorem~2.3]{Dragicevic_Pituk-Duality_between_Shadowing_and_Uniform_Expansivity_in_Linear_Dynamics} guarantees that a bounded operator $T$ on a Banach space $\mc{B}$ has the shadowing property if and only if $\sigma_{sur}(T)\cap\SP^{1}=\emptyset$.

Notice that, for every $\lambda\in\C$, the operator $T_{\bar{d}, \bar{w}}-\lambda I$ is also a perturbed weighted shift. Thus, in view of \eqref{Eq_Shadowing Spectral-Analytic Conditions}, the proof of Theorem~\ref{Theor_Shadowing Petrurbed Weighted Shifts} reduces to determining necessary and sufficient conditions for a perturbed weighted shift to be surjective. To this end, the following Lemma provides a sufficient condition for the existence of a right inverse
$S_{\bar{d}, \bar{w}}:\ell^{p}\lc\N\rc\to\ell^{p}\lc\N\rc$
of $T_{\bar{d}, \bar{w}}$, whose formula is given by \eqref{Eq_Inverse Operator Definition}, as well as a necessary and sufficient condition for $S_{\bar{d}, \bar{w}}$ to be bounded.

\begin{lemma}\label{Lem_Right Inverse Existence}
The operator $S_{\bar d,\bar w} : \ell^{p}\lc \N \rc \to \ell^{p}\lc \N \rc$, given by the formula \eqref{Eq_Inverse Operator Definition}, is well defined and bounded if and only if
\begin{equation}\label{Eq_Right Inverse Bounded Necessary and Sufficient Condition}
\lim_{l \to \infty}\sup_{k\ge 0}
\dfrac{
\lav d_{k+1} \cdots d_{k+l}\rav^{\frac{1}{l}}
}{
\lav w_{k+1} \cdots w_{k+l}\rav^{\frac{1}{l}}
}<1 \cdot 
\end{equation}
Moreover, if $T_{\bar d,\bar w}$ is surjective and non-injective, then condition~\eqref{Eq_Right Inverse Bounded Necessary and Sufficient Condition} holds.
\end{lemma}

\noindent
The proof of Lemma~\ref{Lem_Right Inverse Existence} uses the following well-known result (see, for instance, \cite[Theorem~A.10]{Aiena_Fredholm_Local_Spectral_Theory} and \cite[Theorem~2.12]{Brezis_Functional_Analysis}): for a bounded surjective linear operator $T$ on a Banach space $\mc{B}$,
\begin{equation}\label{Eq_Right Inverse and Closed Complement}
T \text{ admits a bounded linear right inverse if and only if } \ker T \text{ admits a closed complement.}
\end{equation}

\begin{proof}[Proof of Lemma~\ref{Lem_Right Inverse Existence}]
First, assume that \eqref{Eq_Right Inverse Bounded Necessary and Sufficient Condition} holds. In this case, it follows that the constant $c_{\bar{d}, \bar{w}}$ defined in \eqref{Eq_Cor Inverse Operator Norm Bound} is finite. Hence, by Lemma~\ref{Lem_Norm of S}, the operator
$S_{\bar{d}, \bar{w}}: \ell^{p}\lc \N \rc \to \ell^{p}\lc \N \rc$
is well defined and satisfies
$\ldav S_{\bar{d}, \bar{w}} \rdav \le c_{\bar{d}, \bar{w}}$. 
Conversely, suppose that $\ldav S_{\bar d,\bar w}\rdav<\infty$. Then, for every $k\ge1$, the $k$-th vector $\bm{e}_{k}$ of the canonical basis satisfies
$$ 1 + \sum_{l=1}^{\infty}\dfrac{|d_{k+1}.\cdots d_{k+l}|^p}{|w_{k+1}.\cdots w_{k+l}|^p} \le |w_{k} |^p\cdot||S_{\bar d, \bar w}(e_k)||_p^p \le   \sup_{k \ge 1}|w_k |^p\cdot||S_{\bar d,\bar w}||^{p} .$$
Consequently, $ \sup_{k\ge1}
\lc
\sum_{l=1}^{\infty}
\dfrac{\lav d_{k+1}\cdots d_{k+l}\rav^p}
{\lav w_{k+1}\cdots w_{k+l}\rav^p}
\rc
<+\infty$. Hence, relation~\eqref{Eq_Right Inverse Bounded Necessary and Sufficient Condition} follows by a straightforward application of \cite[Lemma~18]{Bernardes_Messaoudi-Shadowing_and_Structural_Stability} to the bounded sequence $\lc \lav d_{k}/w_{k}\rav^p\rc_{k\ge1}$.

\medskip

Assume, for the rest of the proof, that $T_{\bar d,\bar w}$ is surjective and non-injective. By the first part of the Lemma, it suffices to show that the associated right inverse $S_{\bar d,\bar w}$ is bounded. Since
$\ker T_{\bar d,\bar w}=\operatorname{span}\langle\bm v_1\rangle$,
where $\bm v_1$ is defined in \eqref{Eq_Kernel Basis Vectors}, the kernel is one-dimensional and hence admits a closed complement. Thus, \eqref{Eq_Right Inverse and Closed Complement} guarantees the existence of a bounded right inverse $R$ of $T_{\bar d,\bar w}$. Since both $R$ and $S_{\bar d,\bar w}$ are right inverses,
$
R\lc\bm x\rc-S_{\bar d,\bar w}\lc\bm x\rc
\in\ker T_{\bar d,\bar w}
=
\operatorname{span}\langle\bm v_1\rangle$, $\bm x\in\ell^p\lc\N\rc$. 
Hence, there exists a linear functional $\varphi:\ell^p\lc\N\rc\to\C$ such that
$
R\lc\bm x\rc
=
S_{\bar d,\bar w}\lc\bm x\rc
+
\varphi\lc\bm x\rc\bm v_1$, $\bm x\in\ell^p\lc\N\rc$. 
Since the first coordinate of $S_{\bar d,\bar w}\lc\bm x\rc$ is zero, whereas that of $\bm v_1$ is nonzero, $\varphi$ is determined by the bounded functional
$\bm x\mapsto\lc R\lc\bm x\rc\rc_1$.
Thus, $\varphi$ is bounded and, since $R$ is bounded, so is $S_{\bar d,\bar w}$.

The proof is complete.
\end{proof}

\paragraph{}

The proof of Theorem~\ref{Theor_Shadowing Petrurbed Weighted Shifts} is given next. In what follows, $\ell^{\infty}(\mathbb N)$ denotes the Banach space of bounded complex sequences equipped with its usual supremum norm $\ldav\cdot\rdav_{\oo}$. Given $1\le q \le +\oo$ 
 and a bounded sequence of complex numbers $\bar{w} = \lc w_{k} \rc_{k=1}^{+\oo}$ satisfying $\inf_{k \ge 1}\lav w_{k} \rav \neq 0$, let $F_{\bar{w}} : \ell^{q}\lc \N \rc \to \ell^{q}\lc \N \rc $ be the \textit{unilateral forward weighted shift} given by 
\begin{equation}\label{Eq_Def_Forward Weighted Shift}   
F_{\bar w}(x_1,x_2,x_3,\ldots) = (0,w_1x_1,w_2x_2,w_3x_3,\ldots) .
\end{equation} 
The proof will make use of the following description of the approximate spectrum of $F_{\bar{w}}$ (see, for instance, \cite{Miller-Miller-Neumann_weighted_shifts}): 
\begin{equation}\label{Eq_Forward Shift Approximate Spectrum} \sigma_{\mathrm{ap}}(F_{\overline{w}}) = \left\{ \lambda\in\mathbb{C} :\, r_{-}(F_{\overline{w}})\le |\lambda|\le r_{+}(F_{\overline{w}}) \right\},  \end{equation}
where
\begin{equation}\label{Eq_Forward Shift Approximate Spectrum Auxiliary Quantities}
r_{-}(F_{\overline{w}}):=\lim_{l \to \infty} \inf_{k \ge 0}\bigl| w_{k+1}\cdots w_{k+l}\bigr|^{1/l} \; \text{ and } \;   r_{+}(F_{\overline{w}}):=\lim_{l \to \infty} \sup_{k \ge 0}\bigl| w_{k+1}\cdots w_{k+l} \bigr|^{1/l} . 
\end{equation}

\begin{proof}[Proof of Theorem \ref{Theor_Shadowing Petrurbed Weighted Shifts}]
The goal is to establish equivalence \eqref{Eq_Shadowing Spectral-Analytic Conditions}. 
Fix $\lambda \in \SP^{1}$. The proof proceeds by splitting into two cases.

\medskip 
In the first case, assume that
$\lambda\in\SP^1\cap\overline{\lfp d_k\rfp_{k\ge1}}$.
In this case, it follows that $r_{inf}\lc\lambda\rc=0$. Thus, to prove~\eqref{Eq_Shadowing Spectral-Analytic Conditions}, it suffices to show that $T_{\bar d,\bar w}-\lambda I$ is surjective if and only if $r_{sup}\lc\lambda\rc<1$.

To this end, consider the case in which  $T_{\bar d,\bar w}-\lambda I$ is surjective. Upon verifying that 
$\overline{ \lfp d_k \rfp_{k\ge1}}\subseteq\sigma(T_{\bar d,\bar w})$, one infers that the operator
$T_{\bar d,\bar w}-\lambda I$ is non-injective. Therefore, applying Lemma~\ref{Lem_Right Inverse Existence} to $\lc d_k-\lambda\rc_{k=1}^{+\oo}$ and $\lc w_k\rc_{k=1}^{+\oo}$, and using the definition of $r_{sup}\lc\lambda\rc$ in \eqref{Eq_Shadowing Sup Bound}, yields $ r_{sup}\lc\lambda\rc < 1$.
Conversely, suppose that $r_{sup}\lc\lambda\rc<1$. Then, by Lemma~\ref{Lem_Right Inverse Existence}, the right inverse operator $S_{\bar d_{\lambda},\bar w}$ of $T_{\bar d,\bar w}-\lambda I$, where
$\bar d_{\lambda}=\lc d_k-\lambda\rc_{k=1}^{+\oo}$,
is well defined and bounded. Therefore, the operator $T_{\bar d,\bar w}-\lambda I$ is surjective.

\medskip 

In the second case, assume that $\lambda\in\SP^{1}\backslash\overline{\lfp d_k\rfp_{k\ge1}}$. Let
$$
q=
\begin{cases} 
p/\lc p-1\rc, & \text{if } 1<p<+\oo, \\
+\oo, & \text{if } p=1,
\end{cases} 
$$
and denote by $T_{\bar d,\bar w}^{*}:\ell^{q}\lc\N\rc\to\ell^{q}\lc\N\rc$ the adjoint of $T_{\bar d,\bar w}$. By \cite[Theorem~1.10]{Aiena_Fredholm_Local_Spectral_Theory},
$\sigma_{sur}(T_{\bar d,\bar w})=\sigma_{ap}(T_{\bar d,\bar w}^{*})$.
Hence, to establish \eqref{Eq_Shadowing Spectral-Analytic Conditions}, it suffices to show that
\begin{equation}\label{Eq_Forward Shift Approximate Spectrum Analytic Condition} 
\lambda\not\in\sigma_{ap}\lc T_{\bar d,\bar w}^{*}\rc
\quad\text{if and only if}\quad 
1\not\in\lb r_{inf}\lc\lambda\rc,r_{sup}\lc\lambda\rc\rb 
.
\end{equation}

It is readily verified that $T_{\bar d,\bar w}^{*}=I_{\bar d}+F_{\bar w}$, where $F_{\bar w}$ is defined in \eqref{Eq_Def_Forward Weighted Shift}, and 
\begin{equation*}\label{Eq_Factorization of Adjoint Operator}
T_{\bar d,\bar w}^{*}-\lambda I
=
\lc F_{\bar u_\lambda}+I\rc\circ\lc I_{\bar d}-\lambda I\rc,
\qquad
\bar u_\lambda
=
\left(
\frac{w_k}{d_k-\lambda}
\right)_{k\ge1}.
\end{equation*}
Since $\lambda\notin\overline{\lfp d_k\rfp_{k\ge1}}$, the operator $I_{\bar d}-\lambda I$ is invertible, and hence $ \lambda\in\sigma_{ap}\lc T_{\bar d,\bar w}^{*}\rc$ if and only if $ -1\in\sigma_{ap}\lc F_{\bar u_\lambda}\rc$.
Together with \eqref{Eq_Forward Shift Approximate Spectrum} and the identities $r_{-}\lc F_{\bar u_\lambda}\rc
=
1/r_{sup}\lc\lambda\rc$ and $r_{+}\lc F_{\bar u_\lambda}\rc
= 1/ r_{inf}\lc\lambda\rc$, 
this yields \eqref{Eq_Forward Shift Approximate Spectrum Analytic Condition}. The last two identities follow directly from \eqref{Eq_Shadowing Inf Bound}, \eqref{Eq_Shadowing Sup Bound}, and \eqref{Eq_Forward Shift Approximate Spectrum Auxiliary Quantities}.

The proof of the theorem is complete. 
\end{proof}

\paragraph{Acknowledgments} 
\thanks{Ali Messaoudi was partially supported by CNPq grant 310784/2021-2 and by FAPESP projects  2024/15612-6 and 2024/04685-2. 
José Luiz Tofanin Neto was supported by CAPES grant 88887.888179/2023-00. 
Ioannis Tsokanos was supported by FAPESP grant 2024/10135-5. }

\par\bigskip
\begingroup
\footnotesize

\par\bigskip
\begingroup
\footnotesize

\indent\textsc{Departamento de Matemática, Universidade Estadual Paulista,
Rua Cristóvão Colombo, 2265, Jardim Nazareth,
São José do Rio Preto, SP, 15054-000, Brasil.}\par
\indent\textit{E-mail address:} \texttt{ali.messaoudi@unesp.br}\par

\indent\textsc{Departamento de Matemática, Universidade Estadual Paulista,
Rua Cristóvão Colombo, 2265, Jardim Nazareth,
São José do Rio Preto, SP, 15054-000, Brasil.}\par
\indent\textit{E-mail address:} \texttt{jose.tofanin@unesp.br}\par

\indent\textsc{Departamento de Matemática, Universidade Estadual Paulista,
Rua Cristóvão Colombo, 2265, Jardim Nazareth,
São José do Rio Preto, SP, 15054-000, Brasil.}\par
\indent\textit{E-mail address:} \texttt{ioannis.tsokanos@unesp.br}\par

\endgroup


\begin{thebibliography}{99}

\bibitem{Aiena_Fredholm_Local_Spectral_Theory}
P.~Aiena,
\textit{Fredholm and Local Spectral Theory II: With Application to Weyl-Type Theorems},
Lecture Notes in Mathematics, vol.~2235, Springer, Cham, 2018.

\bibitem{Antunes}
M.~B. Antunes, G.~E. Mantovani and R. Var\~ao,
Chain recurrence and positive shadowing in linear dynamics,
\textit{J. Math. Anal. Appl.} \textbf{506} (2022), no.~1, 125622.

\bibitem{Bayart_Grivaux_2006}
F. Bayart and S. Grivaux,
Frequently hypercyclic operators,
\textit{Trans. Amer. Math. Soc.} \textbf{358} (2006), no.~11, 5083--5117.

\bibitem{Bayart-Matheron-Dynamics_Of_Linear_Operators}
F. Bayart and É. Matheron,
\textit{Dynamics of Linear Operators},
Cambridge Tracts in Mathematics, vol.~179, Cambridge University Press, Cambridge, 2009.

\bibitem{Bayart-Ruzsa_difference_sets}
F. Bayart and I.~Z. Ruzsa,
Difference sets and frequently hypercyclic weighted shifts,
\textit{Ergodic Theory Dynam. Systems} \textbf{35} (2015), no.~3, 691--709.


\bibitem{BM}
N.~C. Bernardes Jr. and A. Messaoudi,
Shadowing and structural stability for operators,
\textit{Ergodic Theory Dynam. Systems} \textbf{41} (2021), no.~4, 961--980.

\bibitem{Bernardes_Messaoudi-Shadowing_and_Structural_Stability}
N.~C. Bernardes Jr. and A. Messaoudi,
Shadowing and structural stability in linear dynamical systems,
arXiv:1902.04386 [math.DS], (2019).

\bibitem{Bernardes_Peris-On_Shadowing_and_Chain_Recurrence_in_Linear_Dynamics}
N.~C. Bernardes Jr. and A. Peris,
On shadowing and chain recurrence in linear dynamics,
\textit{Adv. Math.} \textbf{441} (2024), 109539.

\bibitem{Bourhim-Spectrum_of_Bilateral_Shifts}
A. Bourhim,
Spectrum of bilateral shifts with operator-valued weights,
\textit{Proc. Amer. Math. Soc.} \textbf{134} (2006), 2131--2137.

\bibitem{Brezis_Functional_Analysis}
H.~Brezis,
\textit{Functional Analysis, Sobolev Spaces and Partial Differential Equations},
Universitext, Springer, New York, 2011.

\bibitem{Cirilo_Gollobit_Pujals-Dynamics_of_Generalized_Hyperbolic_Linear_Operators}
P.~R. Cirilo, B. Gollobit and E. Pujals,
Dynamics of generalized hyperbolic linear operators,
\textit{Adv. Math.} \textbf{387} (2021), 107830.

\bibitem{Dragicevic_Pituk-Duality_between_Shadowing_and_Uniform_Expansivity_in_Linear_Dynamics}
D. Dragičević and M. Pituk,
Duality between shadowing and uniform expansivity in linear dynamics,
to appear in \textit{Trans. Amer. Math. Soc.}

\bibitem{Grivaux-Sums_of_hypercyclic_operators}
S. Grivaux,
Sums of hypercyclic operators,
\textit{J. Funct. Anal.} \textbf{202} (2003), no.~2, 486--503.

\bibitem{Gr2}
S. Grivaux,
Hypercyclic operators, mixing operators, and the bounded steps problem,
\textit{J. Operator Theory} \textbf{54} (2005), no.~1, 147--168.

\bibitem{Grivaux_Matheron_Menet-Linear_Dynamical_Systems_on_Hilbert_Spaces}
S. Grivaux, É. Matheron and Q. Menet,
\textit{Linear Dynamical Systems on Hilbert Spaces: Typical Properties and Explicit Examples},
Mem. Amer. Math. Soc., vol.~269, no.~1315, 2021, 147~pp.

\bibitem{GE}
K.-G. Grosse-Erdmann,
Hypercyclic and chaotic weighted shifts,
\textit{Studia Math.} \textbf{139} (2000), no.~1, 47--68.

\bibitem{Hilden_Wallen-Some_Cyclic_and_Non-Cyclic_Vectors_of_Certain_Operators}
H.~M. Hilden and L.~J. Wallen,
Some cyclic and non-cyclic vectors of certain operators,
\textit{Indiana Univ. Math. J.} \textbf{23} (1974), no.~7, 557--565.

\bibitem{Kitai_invariant_closed_sets}
C.~Kitai,
\textit{Invariant Closed Sets for Linear Operators},
Ph.D. thesis, University of Toronto, Toronto, 1982.

\bibitem{Macdonald-Symmetric_Functions_and_Hall_Polynomials} I.~G. Macdonald, \textit{Symmetric Functions and Hall Polynomials}, 2nd ed., Oxford Mathematical Monographs, Oxford University Press, Oxford, 1995.

\bibitem{Megginson-An_Introduction_to_Banach_Space_Theory}
R.~E. Megginson,
\textit{An Introduction to Banach Space Theory},
Graduate Texts in Mathematics, vol.~183, Springer-Verlag, New York, 1998.

\bibitem{Messaoudi_Neto_Saavedra_Tsokanos-On_GH_S_SH_for_Linear_Operators}
A. Messaoudi, J.~T. Neto, M. Saavedra and I. Tsokanos,
On generalized hyperbolicity, stability, and shadowing for linear operators,
arXiv:2608.17021 [math.DS], (2026).

\bibitem{Miller-Miller-Neumann_weighted_shifts}
T.~L. Miller, V.~G. Miller and M.~M. Neumann,
Local spectral properties of weighted shifts,
\textit{J. Operator Theory} \textbf{51} (2004), 71--88.


\bibitem{SR}
F.~L. Saavedra and A.~M. Rodriguez,
Linear structure of hypercyclic vectors,
\textit{J. Funct. Anal.} \textbf{148} (1997), 524--545.

\bibitem{Salas-Hypercyclic_Weighted_Shifts}
H.~N. Salas,
Hypercyclic weighted shifts,
\textit{Trans. Amer. Math. Soc.} \textbf{347} (1995), no.~3, 993--1004.

\bibitem{Salas-Supercyclicity_and_Weighted_Shifts}
H.~N. Salas,
Supercyclicity and weighted shifts,
\textit{Studia Math.} \textbf{135} (1999), 55--74.

\end{thebibliography}
\end{document}